\documentclass[12 pt]{amsart}
\usepackage{hyperref}
\usepackage{adjustbox}
\usepackage{etex}
\usepackage[shortlabels]{enumitem}
\usepackage{amsmath}
\usepackage{amsxtra}
\usepackage{amscd}
\usepackage{amsthm}
\usepackage{amsfonts}
\usepackage{amssymb}
\usepackage{eucal}
\usepackage[all]{xy}
\usepackage{graphicx}
\usepackage{tikz-cd}
\usepackage{mathrsfs}
\usepackage{subfiles}
\usepackage{mathpazo}
\usepackage[colorinlistoftodos, textsize=tiny]{todonotes}
\usepackage{morefloats}
\usepackage{pdfpages}
\usepackage{thm-restate}
\usepackage[utf8]{inputenc}
\usepackage{epigraph}
\usepackage{csquotes}
\usepackage{stmaryrd}
\usepackage{subfiles}
\usepackage[margin=1in]{geometry}
\graphicspath{ {images/} }

\RequirePackage{color}
\definecolor{myred}{rgb}{0.75,0,0}
\definecolor{mygreen}{rgb}{0,0.5,0}
\definecolor{myblue}{rgb}{0,0,0.65}

\hypersetup{citecolor=blue}
\usepackage{tikz}
\usetikzlibrary{matrix,arrows,decorations.pathmorphing}

\theoremstyle{plain}
\newtheorem{theorem}[subsection]{Theorem}
\newtheorem{proposition}[subsection]{Proposition}
\newtheorem{lemma}[subsection]{Lemma}

\theoremstyle{definition}
\newtheorem{definition}[subsection]{Definition}
\newtheorem{remark}[subsection]{Remark}
\newtheorem{example}[subsection]{Example}

\newtheorem{question}[subsection]{Question}
\newtheorem{conjecture}[subsection]{Conjecture}

\theoremstyle{remark}
\newtheorem{notation}[subsection]{Notation}

\numberwithin{equation}{section}
\newcommand\nc{\newcommand}
\nc\on{\operatorname}
\nc\renc{\renewcommand}

\newcommand\bp{\mathbb P}

\newcommand\bz{\mathbb Z}

\newcommand\scc{\mathscr C}

\newcommand\sce{\mathscr E}

\newcommand\scl{\mathscr L}

\newcommand\sco{\mathscr O}

\newcommand\scv{\mathscr V}
\newcommand\scw{\mathscr W}
\newcommand\scx{\mathscr X}

\newcommand \ra{\rightarrow}

\DeclareMathOperator\spec{\text{Spec}}

\newcommand*{\shom}{\mathscr{H}\kern -.5pt om}
\newcommand*{\stor}{\mathscr{T}\kern -.5pt or}
\newcommand*{\sext}{\mathscr{E}\kern -.5pt xt}

\makeatletter
\newcommand{\customlabel}[2]{\protected@write \@auxout {}{\string \newlabel {#1}{{#2}{\thepage}{#2}{#1}{}} }\hypertarget{#1}{#2}}

\RequirePackage{color}
\definecolor{myred}{rgb}{0.75,0,0}
\definecolor{mygreen}{rgb}{0,0.5,0}
\definecolor{myblue}{rgb}{0,0,0.65}

\usepackage{color}

\newcommand \m{\overline{\mathscr M}_{1,1}}
\DeclareMathOperator\height{ht}
\DeclareMathOperator\st{st}

\DeclareMathOperator\id{id}

\DeclareMathOperator\coker{coker}

\DeclareMathOperator\gl{GL}

\DeclareFontFamily{U}{wncy}{}
\DeclareFontShape{U}{wncy}{m}{n}{<->wncyr10}{}
\DeclareSymbolFont{mcy}{U}{wncy}{m}{n}
\DeclareMathSymbol{\Sha}{\mathord}{mcy}{"58}

\def\listtodoname{List of Todos}
\def\listoftodos{\@starttoc{tdo}\listtodoname}

\title{Stacky heights on elliptic curves in characteristic 3}
\author{Aaron Landesman}
\address{Dept. of Mathematics, Harvard University,
	Cambridge, MA 02138} 
\email{aaronlandesman@gmail.com}

\usepackage{microtype}

\begin{document}

\maketitle
\begin{abstract}
	We show there are no stacky heights on the moduli stack of stable elliptic curves in characteristic
	$3$ which induce the usual Faltings height, negatively answering a
	question of Ellenberg, Satriano, and Zureick-Brown.
\end{abstract}

\section{Introduction}
\label{section:intro-stack-heights}

In this paper, we investigate heights on the compactified moduli stack of
elliptic curves in characteristic $3$.
We show that the notion of stacky height introduced in
\cite{ellenbergSZB:heights-on-stacks} 
does not always
recover the classical notion of height.
Specifically, we show there is no vector bundle whose associated stacky height 
induces the usual notion of Faltings height for elliptic curves in
characteristic $3$.

Throughout this paper, we work over a fixed perfect field $k$ of characteristic $3$.
Let $\m$ denote the Deligne-Mumford moduli stack of stable elliptic curves over $k$.
Given a finite field extension $K$ over $k(t)$ and an elliptic curve $E \ra \spec K$, there is a Faltings height on that elliptic curve, which we define as follows.
\begin{definition}
	\label{definition:height}
	Given $E \to \spec K$ as above, let $C$ be the regular proper connected curve over $k$ whose generic point is $\spec K$
	and let $f: X \ra C$ denote the minimal proper regular model of
	$E \ra \spec K$. The {\em Faltings height of $E$} is given by $\deg f_* \omega_{X/C}$.
Thinking of $E \ra \spec K$ as a $K$ point of $\m$, given by $x: \spec K \ra \m$, we denote this Faltings height by $\height(x)$.
\end{definition}
\begin{remark}
	\label{remark:}
	The Faltings height is also computable by the formula $\frac{1}{12}\deg(\Delta_{X/C})$ where $\Delta_{X/C}$ is the discriminant of the relative elliptic
surface, viewed as a section of $H^0(C, f_* \omega_{X/C}^{\otimes 12})$.
\end{remark}

On the other hand, suppose we are given a vector bundle $\scv$ on $\m$ and a $K$ point $x:
\spec K \ra \m$, corresponding to a stable elliptic curve
$E \ra \spec K$.
Ellenberg, Satriano, and Zureick-Brown
\cite[Definition 2.11]{ellenbergSZB:heights-on-stacks}
define a notion of height associated to $x$ and $\scv$, notated
$\height_\scv(x)$,
see \autoref{definition:stacky-height}.
Suppose $k'$ is a field of characteristic not $2$ or $3$, and let $\omega$ denote the Hodge bundle over $(\overline{\mathcal M}_{1,1})_{k'}$.
That is, $\omega := f_* \omega_{\sce/(\overline{\mathcal M}_{1,1})_{k'}}$ for
$f: \sce \ra (\overline{\mathcal M}_{1,1})_{k'}$ the universal stable elliptic curve.
Then
\cite[Proposition 3.11]{ellenbergSZB:heights-on-stacks} show
that, for $K$ a finite extension of $k'(t)$, and $x: \spec K \ra (\overline{\mathcal M}_{1,1})_{k'}$ a point,
$\height_{\omega}(x) = \height(x)$, with the latter notion of Faltings height as defined in \autoref{definition:height}.
However, as 
\cite[p. 27]{ellenbergSZB:heights-on-stacks}
observe, for $k$ a field of characteristic $3$, it is no longer true
that 
$\height_{\omega}(x) = \height(x)$
for all $x: \spec K \ra \m$.
In particular, they show cubic twists of the form $y^2 = x^3 - x + f(t)$, for
$f(t) \in k[t]$, all have height $0$ with
respect to the Hodge bundle, even though their Faltings heights can be nonzero.

Moreover, they show there is no line bundle $\scl$ on $\m$ for which
$\height_{\scl}(x) = \height(x)$
\cite[p. 27]{ellenbergSZB:heights-on-stacks}.
This leads to the following question:
\begin{question}
	\label{question:}
	Is there some vector bundle $\scv$ (necessarily of rank more than $1$) on $\m$ so that $\height_\scv(x) = \height(x)$ for every $x: \spec K \ra \m$ over a field of characteristic $2$ or $3$?
\end{question}

In this note, we show that the answer is ``no'' when $k$ is a perfect field of characteristic $3$.
More precisely, we have the following:

\begin{theorem}
	\label{theorem:no-height}
	Let $\m$ denote the Deligne-Mumford stack of stable elliptic curves over
	a perfect field $k$ of characteristic $3$.
	There is no vector bundle $\scv$ on $\m$ for which $\height_\scv(x) =
	\height(x)$ for all points $x: \spec K \ra \m$, where $K$ is a finite extension of $k(t)$.
\end{theorem}

We deduce this from \autoref{theorem:local} in
\autoref{subsection:proof-no-height}.

This result leaves open the question as to whether 
there some vector bundle $\scw$ on $\m$, over a field of characteristic $3$, and some integer $n$ so that
$n\height(x) = \height_\scw(x)$.
We conjecture the answer is no:

\begin{conjecture}
	\label{conjecture:}
	There is no vector bundle $\scw$ on $\m$, over a field of characteristic
	$3$, for which there exists an
	integer $n$ such that
	$n\height(x) = \height_\scw(x)$.
\end{conjecture}
See \cite[Remark 9.2.7 and 9.2.8]{landesman:thesis} for some speculation related
to this conjecture.

Another related question, originally posed to us by Jordan Ellenberg, is whether there exist Northcott stacky heights on $\m$
in characteristic $3$.
We say a height function on the set of $k(t)$ points of a stack $\scx$ satisfies
the {\em Northcott property} if there are
finitely many such points of bounded height, cf. \cite[p.
4]{ellenbergSZB:heights-on-stacks}.
\begin{question}[Ellenberg]
	\label{question:}
	Does there exist a vector bundle $\scv$ on $\m$ over a finite field of
	characteristic $3$ whose induced height function $\height_\scv$ is
	Northcott?
\end{question}
\begin{remark}
	\label{remark:}
	In \cite[Theorem 9.2.4]{landesman:thesis}, I had claimed there do exist
	such Northcott bundles. However, the proof of this relies on \cite[Lemma
	9.7.1]{landesman:thesis}, which contains an error
	where I incorrectly claimed that trigonal curves have a certain minimal form without justification.
	This leaves the above question open.
\end{remark}

\subsection{Idea of the proof of \autoref{theorem:no-height}}
\label{subsection:idea-of-no-stacky-height-proof}
We use notation for Kodaira reduction type, as pictured in \cite[p.
365]{Silverman:Advanced}.
The idea of the proof of \autoref{theorem:no-height} is to show that any $\scv$ which induces the correct local
stacky height for places of type Kodaira $\mathrm{III}$ reduction necessarily induces the incorrect local stacky
height for cubic twists.

We now elaborate on the above idea.
The substack $BG \subset \m$ corresponding to elliptic curves with $j$-invariant $0$ has
geometric automorphism group $G$, where $G$ is the dicyclic group of order $12$.
When we restrict $\scv$ to $BG$, we obtain a $G$-representation $\rho$.
We show that some element $g\in G$ of order $4$ acts with a codimension $1$
fixed space and no eigenvalues equal to $-1$.
This is enough to deduce that $\rho$ is a sum of $1$-dimensional
representations, and hence factors through the abelianization of $G$.
We then show that any such vector bundle 
cannot detect nontrivial stacky heights associated to 
elliptic curves which are isotrivial
cyclic cubic twists.

\subsection{Overview}
The structure of this paper is as follows.
We review the notion of stacky heights in \autoref{section:eszb-review}.
We then reduce \autoref{theorem:no-height} to the statement about local stacky
heights, \autoref{theorem:local}, at the end of \autoref{section:reduction}.
Finally, we prove \autoref{theorem:local} at the end of
\autoref{section:computing-heights}, using a group theoretic input from
\autoref{section:s3-representations}.

\subsection{Acknowledgements}
I thank an extremely helpful referee for catching a number of egregious errors
in an earlier version.
I thank Jordan Ellenberg for bringing the questions addressed in this paper to
my attention,
and for many useful discussions.
I also thank Dori Bejleri, Pavel Etingof, Johan de Jong, Anand Patel, Matt Satriano, Ravi Vakil, 
Takehiko Yasuda,
and David Zureick-Brown
for helpful conversations.

\section{Review of the definition of heights on stacks}
\label{section:eszb-review}

We now recall the definition of heights on stacks introduced in
\cite{ellenbergSZB:heights-on-stacks}.

\begin{definition}
	\label{definition:tuning-stack}
	Let $k$ be a field, let $C$ be a regular proper integral curve over $k$,
	and let $K := K(C)$.
	Let $\scx$ be an algebraic stack over $k$ and
	$x: \spec K  \to \scx$ be a $K$-point.
	A {\em tuning stack}
	$\scc$
	for $x$ is a normal algebraic stack $\scc$ with finite diagonal
	together with a map
	$\overline{x}: \scc \to \scx$
	extending $x$ so that $\pi: \scc \to C$ is a birational coarse
	space map.
	A tuning stack $(\scc, \overline{x},\pi)$ is a universal tuning stack if it is terminal among all
	all tuning stacks.
\end{definition}

\begin{remark}
	\label{remark:number-field-tuning-stack}
	Although we will not need it,
	one can also extend \autoref{definition:tuning-stack}
to the number field case as
follows.
Let $L$ be a number field, $B = \spec \sco_L$, and let $\scx$ be an algebraic stack
over $B$. Let $K/L$ be a finite extension of number fields and $x: \spec K \to \scx$ be a $K$
point.
A tuning stack $\scc$ for $x$ is then a 
normal algebraic stack $\scc$ with finite diagonal
together with a map
$\overline{x}: \scc \to \scx$
extending $x$ so that $\pi: \scc \to \spec \sco_K$ is a birational coarse
space map.
\end{remark}

\begin{remark}
	\label{remark:}
	This definition of tuning stack differs slightly from that of 
	\cite[Definition 2.1]{ellenbergSZB:heights-on-stacks},
in that we do not require $\scx$ to be a stack over $C$. 
In particular, if $\scx$ is a stack over a base $k$,
we may discuss
tuning stacks and heights of points associated to function fields of
transcendence degree $1$ over
$k$.
\end{remark}

We now recall the notion of stacky height.
\begin{definition}[~\protect{\cite[Definition
	2.11]{ellenbergSZB:heights-on-stacks}}]
	\label{definition:stacky-height}
	Let $\scx$ be a proper algebraic stack over a field $k$ with finite
	diagonal, let $C$ be a smooth
	proper connected curve over $k$ with function field
	$K(C)$.
	Let $\scv$ be a vector bundle on $\scx$ and $x \in \scx(K(C))$.
	If $\scc$ is a tuning stack for $x$
	and $\overline{x}, \pi$ are the corresponding maps defined in
	\autoref{definition:tuning-stack},
	then we define the {\em height of $x$ with respect to $\scv$} as
	\begin{align*}
		\height_\scv(x) := -\deg_C\left( \pi_* \overline{x}^* \scv^\vee
		\right).	
	\end{align*}
	We also define the {\em stable height of $x$ with respect to $\scv$} as
	\begin{align*}
		\height^{\st}_\scv(x) := -\deg_\scc\left(\overline{x}^* \scv^\vee
		\right).	
	\end{align*}
\end{definition}

To make sense of this definition, one has to check various properties. 
For example, one must verify this notion is independent of
the choice of tuning stack. These are verified in
\cite[\S2.2]{ellenbergSZB:heights-on-stacks}.

One can also define heights in the number field case, but then one has to use
metrized line bundles and Arakelov heights as in \cite[Appendix
A]{ellenbergSZB:heights-on-stacks}. We will only be concerned with the
function field case, and so do not discuss this further.

Finally, we recall the notion of local stacky height, which will play a crucial
role in our proof.
\begin{notation}
	\label{notation:coarse}
	Let $L$ be a field and let $K$ either be a number field or a finite
	extension of $L(t)$ with regular model $C$. Here, $C$ is the
		spectrum of the ring of integers $\sco_K$ in the case $K$ is a number field and a regular proper curve over $L$ when $K$ is 
	a function field.
	Suppose $x: \spec K \ra \scx$ is a $K$ point of a stack $\scx$ locally of finite presentation and with finite diagonal, so that $\scx$ possess a coarse moduli space
	$\alpha: \scx \ra X$.
	Let $\scc$ denote the universal tuning stack associated to $x$, as in
	\autoref{definition:tuning-stack} or
	\autoref{remark:number-field-tuning-stack}.
	Let $C_v$ denote the localization of $C$ at $v$ and let 
	$\mathscr C_v := \mathscr C \times_C C_v$.
	Then, we have a diagram
	\begin{equation}
		\label{equation:}
		\begin{tikzcd} 
			\spec K \arrow[ddr, bend right=20] \ar{dr} \ar[bend
			left, drr,swap, "x"] && \\
			\mathscr C_v \ar[crossing over]{r}\ar{d} & \scc \ar {r}{\overline x} \ar {d}{\pi} & \scx \ar {d}{\alpha} \\
			C_v \ar{r} & C \ar {r} & X.
	\end{tikzcd}\end{equation}
	In this setting, for $\scv$ a vector bundle on $\scx$ and $v$ a place of
	$K$, we recall the definition of the 
	local stacky height, which is equivalent to that given in
	\cite[Definition 2.21]{ellenbergSZB:heights-on-stacks}.
	The
	{\em local stacky height} associated to
	$x$ and $\scv$ at the place $v$ as
	\begin{align*}
		\delta_{\scv;v}(x) := \deg (\coker\left(\gamma^* \pi^*\pi_*
		\overline x^*\scv^\vee \ra  \gamma^* \overline x^*
\scv^\vee\right)).
	\end{align*}
\end{notation}

\section{The reduction to $j=0$}
\label{section:reduction}

In this section, we explain how to reduce \autoref{theorem:no-height} to another
result about the structure of our vector bundle restricted to a particular
residual gerbe of $\m$. 
In order to prove our main result, we will need to examine elliptic curves with
extra automorphisms. The next remark describes them.
\begin{remark}
	\label{remark:dicyclic}
Recall that there are only $2$ possibilities for the geometric automorphism group of an elliptic curve in characteristic $3$.
It is either $\bz/2\bz$ or the dicyclic group of order $12$, which we denote
$G$. 
The following basic facts about $G$ can be found in \cite{groupnames-Dic3}.
This dicyclic group $G$ of order $12$ is a semidirect product $G \simeq
\bz/3\bz \rtimes \bz/4\bz$
where a generator of $\bz/4\bz$ acts on $\bz/3\bz$ by the nontrivial automorphism.
For the remainder, we fix a splitting to identify $\bz/4\bz$ as a subgroup of $G$.

It will be useful to note that the center of $G$ is $\bz/2\bz$, which can be identified as an index $2$ subgroup of $\bz/4\bz$ for any choice of splitting $\bz/4\bz \ra G$.
The quotient of $G$ by its central $\bz/2\bz$ is isomorphic to $S_3$.

Further, a semistable elliptic curve has geometric automorphism group $G$ if and only if its $j$ invariant is $0$.
\end{remark}

The key to proving \autoref{theorem:no-height} is the following:
\begin{theorem}
	\label{theorem:local}
	Suppose $k$ is an algebraically closed field of characteristic $3$.
	Suppose $\scv$ is a vector bundle on $\m$ for which $\height_\scv(x) =
	\height(x)$ for all points $x: \spec K \ra \m$, for $K$ ranging over
	finite extensions of $k(t)$.
	Let $\iota: [\spec k/G]\to \m$ be the residual gerbe over the point
	of $j$-invariant $0$, with
	$G$ as in \autoref{remark:dicyclic}.
	If $\gamma : [\spec k/(\mathbb Z/3 \mathbb Z)] \to [\spec k/G]$ denotes
	the quotient by $\mathbb Z/4 \mathbb Z$, then $\gamma^* \iota^* \scv$ is
	trivial.
\end{theorem}
We prove this in \autoref{subsection:proof-local}.

\subsection{Deducing \autoref{theorem:no-height}}
We next verify 
\autoref{theorem:no-height} assuming \autoref{theorem:local}.

\subsection{\autoref{theorem:local} implies \autoref{theorem:no-height}}
\label{subsection:proof-no-height}
To prove \autoref{theorem:local}, that no vector bundle can induce a stacky
height which agrees with Faltings height, We first observe we may assume $k$ is algebraically closed.
Indeed, both Faltings height and stacky height are preserved under base change
along extensions of $k$.

	By \autoref{theorem:local}, any point $x: \spec K \to \m$ factoring through $B(\bz/3\bz)$ at
	the point of $\m$ corresponding to $j$-invariant $0$ must have
	$\on{ht}_\scv(x) = 0$.
	Therefore, it suffices to construct an elliptic curve over $\spec k(t)$
	so that the associated map $x: \spec k(t) \to \m$ factors through
	$\spec k(t) \to [\spec k/(\mathbb Z/3\mathbb Z)] \to [\spec k/G] \to \m$
	but so that $x$ has nontrivial Faltings height.
	Indeed, we can
	easily construct cubic twists of the form $y^2 = x^3 -x + f(t)$
	with nontrivial Faltings height.
	As a concrete example, we can take $f(t) = t + t^4$ which has additive
	reduction at infinity and Faltings height $1$, as can be verified with a
	computer.
	It is therefore enough to verify the associated map $\spec k(t) \to \m$ factors through $[\spec k/(\mathbb
	Z/3 \mathbb Z)]$.
	Indeed, the isotrivial elliptic curve $y^2 = x^3 -x + f(t)$
	becomes trivial over the $\mathbb Z/3\mathbb Z$-extension $k(t)[v]/(v^3 - v -f(t))$, as then we
	can substitute $x-v$ for $x$ to obtain 
\begin{equation*}
	y^2 = (x-v)^3 -(x+v) + f(t)
	= x^3 -x - ( v^3 - v) + f(t)= x^3 - x - f(t) + f(t) = x^3 - x.\qed
\end{equation*}

\section{Representations of $G$ in characteristic $3$}
\label{section:s3-representations}

Throughout this section, we work over a field $k$ of characteristic $3$
containing all $4$th roots of unity.
Let $G$ denote the dicyclic group of order $12$, as described in
\autoref{remark:dicyclic}.
In order to prove \autoref{theorem:local}, we will need to analyze
$G$-representations in characteristic $3$.
The only result in this section we will use in the proof of
\autoref{theorem:local} is \autoref{proposition:abelianization-factor}.
To begin, we show we can decompose any $G$-representation into two
subrepresentations, depending on how the center of $G$ acts.

\begin{lemma}
	\label{lemma:2-summands}
	Let $V$ be a $G$ representation over a base field $k$ of characteristic
	$3$ containing $4$th roots of
	unity.
	Then $V$ splits as a direct sum $V_+ \oplus V_-$ where $V_+ \subset V$
	is the subspace on which the nontrivial central
	element $\alpha \in G$ acts by $\id$ and $V_- \subset V$ is the subspace
	on which $\alpha$ acts by $-\id$.
\end{lemma}
\begin{proof}
	This follows from a standard averaging trick. Namely,
	It is enough to realize $V_1$ and $V_2$ as subrepresentations of $V$.
	Since $2$ is invertible on the base and $\alpha$ has order $2$,
	$\rho(\alpha)$ is a diagonalizable matrix, so $V_1 \oplus V_2 = V$.

	To check $V_1$ and $V_2$ are subrepresentations, it suffices to show
	that for any $v \in V_i$ and any $g \in G$, $\rho(g) v \in V_i$.
	We first check this for $i = 1$.
	Since $\rho(\alpha)|_{V_1} = \id$ and $\rho(\alpha)|_{V_2} = -\id$, we have
$\frac{1}{2} \left( \id + \rho(\alpha)\right)$ is the projector $V \to V_1$ which acts
as the identity on $V_1$. This projector is well defined as $2$ is invertible on the base.
Then for any $v \in V_1$, using that $\alpha$ is central in $G$,
\begin{align*}
\rho(g) v &= \rho(g) \left( \frac{1}{2} \left( \id + \rho(\alpha)\right) v
\right) = \left( \frac{1}{2} \left( \id + \rho(\alpha)\right) \right) (\rho(g) v) \in
V_1
\end{align*}
The case $i = 2$ is completely analogous, using the projector 
$\frac{1}{2} \left( \id - \rho(\alpha) \right)$ in place of
$\frac{1}{2} \left( \id + \rho(\alpha) \right)$.
\end{proof}
Using the surjection $G \to S_3$, we will want the following elementary fact about $S_3$
representations.
\begin{lemma}
\label{lemma:2-eigenvalues}
Let $\rho: S_3 \to \on{GL}(V)$ be a indecomposable representation over a field $k$
of characteristic $3$,
with dimension
at least $2$. Let $\tau \in S_3$ be a transposition. Then, $\rho(\tau)$ has at
least two distinct eigenvalues.
\end{lemma}
\begin{proof}
Because $\tau$ has order $2$, $\rho(\tau)$ is semisimple, and all eigenvalues
are $\pm 1$.
Therefore, it is enough to show $\rho(\tau)$ cannot be $\pm \id$.
After possibly tensoring with the sign representation, it is enough to show $\rho(\tau)
\neq \id$.
Since transpositions generate $S_3$, $\rho(\tau) = \id$ implies $\rho$ is
trivial. Since $\dim \rho > 1$, it is not indecomposable.
\end{proof}

\begin{remark}
\label{remark:}
According to \cite[p. 160]{erdmannH:algebras-and-representation-theory},
there are precisely $6$ indecomposable representations of
$S_3$ in characteristic $3$: the trivial representation, the standard
representation, the permutation representation, and those three representations
tensored with the sign representation, and one can also deduce
\autoref{lemma:2-eigenvalues} directly from this classification.
I had claimed to give an alternate proof of this classification in
\cite[Theorem 9.8.2]{landesman:thesis}, though the proof there has a number of errors.
\end{remark}

Combining the above lemmas, we next show that if an order $4$ element has
$\eta$ as an eigenvalue under $\rho$, it also has $-\eta$ as an eigenvalue.
\begin{lemma}
\label{lemma:negative-eigenvalue}
Let $\rho$ be an indecomposable $G$-representation of dimension at least $2$
over a field $k$ of characteristic $3$ containing $4$th roots of unity, and $g
\in G$ has order $4$. If $\eta$ is a primitive fourth root of unity which is an eigenvalue of $\rho(g)$,
then so is $-\eta$.
\end{lemma}
\begin{proof}
Let $\alpha = g^2$ so that $\alpha$ generates the center of $G$.
By \autoref{lemma:2-summands}, we have that either $V = V_+$ or $V = V_-$ where
$\alpha$ acts on $V_+$ by $\id$ and $\alpha$ acts on $V_-$ by $-\id$.
After tensoring $\rho$ with a $1$ dimensional representation in which $g$ acts by
$\eta$, we may assume $V = V_+$.
Since $\alpha$ generates the kernel of the surjection $G \to S_3$ and $\alpha$
acts trivially on $V = V_+$, $\rho$ factors through $S_3$.
By \autoref{lemma:2-eigenvalues}, $\rho(g)$ must have both $\pm 1$ as an
eigenvalue, as we wished to show.
\end{proof}

We can now use the above lemma to deduce our main group-theoretic result for
$G$-representations.
\begin{proposition}
\label{proposition:abelianization-factor}
Suppose $\rho: G \to \on{GL}(V)$ is a $G$-representation over a field $k$ of characteristic $3$ containing $4$th roots of unity.
If there is an order $4$ element $g \in G$ so that $\rho(g)$ has a codimension
$1$ eigenspace with eigenvalue $1$
and the remaining eigenvalue is a primitive fourth root of unity $\eta$,
then $\rho$ is a sum of $1$-dimensional representations. In particular, $\rho :
G\to \on{GL}(V)$ factors through the abelianization of $G$, $\rho: G \to \mathbb
Z/4\mathbb Z \to \on{GL}(V)$.
\end{proposition}
\begin{proof}
By \autoref{lemma:negative-eigenvalue}, if $\rho$ is any
indecomposable representation of dimension at least $2$, then if a primitive
fourth root of unity $\eta$ shows
up as an eigenvalue of $\rho(g)$, so does $-\eta$. 
Therefore, $\rho$ is a sum of $1$-dimensional representations. 
Hence $\rho$ factors through the abelianization of $G$.
\end{proof}

\section{Review of local Faltings heights}
\label{section:local-heights}
In order to prove our main result, we will need the notion of local Faltings
heights, which we now briefly review.

For $K$ a finite extension of $k(t)$ and $v$ a closed point of the proper regular model of $K$ over $k$
and $x: \spec K \ra \m$, we let $\delta_{\scv;v}(x)$ denote the
local stacky height associated to $\scv$ and $x$ at $v$, as defined in
\autoref{notation:coarse}.
Further, 
\cite{ellenbergSZB:heights-on-stacks}
define a notion of stable stacky height $\height^{\st}_\scv(x)$,
see \autoref{definition:stacky-height}
which satisfies the relation $\height^{\st}_\scv(x) + \sum_v \delta_{\scv; v}(x) = \height_\scv(x)$.
We next recall the analogous notion of local and stable Faltings height associated to elliptic curves.
\begin{definition}
	\label{definition:local-height}
	Let $x: \spec K \ra \m$ be a point.
	For each closed point $v$, let $\spec L \ra \spec K$ be a finite extension over which $x$ acquires semistable reduction
	at $v$.
	Define the local stable Faltings height of $x$ at $v$, notated $\height_v^{\st}(x)$, to be
	$\height_v^{\st}(x) := \frac{1}{12}\sum_{w \mid v} \frac{1}{\deg(L/K)} \deg(\Delta_w)$,
	for $\Delta_w$ the discriminant of $w$ restricted to the local ring at $w$,
	for $w$ ranging over closed points of $L$ over $v$.
	Define the {\em stable Faltings height} by $\height^{\st}(x) := \sum_v \height^{\st}_v(x)$.

	Also, define the {\em local Faltings height} of $x: \spec K \ra \m$ at a closed point $v$, notated $\height_v(x)$, to be $\frac{1}{12}(\deg \Delta_v) - \height^{\st}_v(x)$.
\end{definition}
Using uniqueness of semistable models, one may verify the above notion of stable Faltings height is well defined, and can be computed explicitly in terms of the discriminant
at various closed points $v$.
In what follows, we use Kodaira's notation for reduction type of elliptic
curves. See, for example,
\cite[IV \S9]{Silverman:Advanced}, especially the chart on \cite[p.
365]{Silverman:Advanced}.
\begin{example}
	\label{example:}
	When the fiber at a $k$-rational closed point $v$ has reduction type $I_n^*$, we have $\deg \Delta_v= n+6$
	by \cite[p. 365]{Silverman:Advanced}.
	The valuation of the $j$-invariant of a given Kodaira reduction type is
	$\geq 0$ if and only if that curve has potentially good reduction after
	a base change, and
	is equal to $-n$ if that curve has $I_n$ reduction after a base change.
	From \cite[p. 365]{Silverman:Advanced}, we therefore find that $\height_v^{\st}(x) = \frac{1}{12}n$, 
	and hence $\height_v(x) =  \frac{1}{12}(n + 6) - \frac{1}{12}n = 1/2$.
\end{example}

\section{Computing heights on $\m$}
\label{section:computing-heights}

To conclude the proof of \autoref{theorem:no-height}, 
it remains to prove \autoref{theorem:local}. We do so at the end of this section. 
The basic idea is to compute what the restriction of $\scv$
to $BG$ has to be using \autoref{proposition:abelianization-factor} and
studying the action of a certain order $4$ element of $G$.
This is done via \autoref{lemma:character-values}, whose hypothesis is verified
in \autoref{lemma:local-iii-reduction}.

\begin{remark}
	\label{remark:semistable-reduction}
	If we have an elliptic curve with reduction $\mathrm{II}, \mathrm{III},
	\mathrm{IV}, \mathrm{II^*}, \mathrm{III^*},$ or $\mathrm{IV^*}$ at $v$,
	it has potentially good reduction, with corresponding $j$
	invariant $0$ by \cite[Theorem
	2.1]{miyamotoT:reduction-of-elliptic-curves-in-equal-characteristic-3}.
	Therefore, if $x: \spec K \ra \m$ is a map with one of the above reduction types at $v$, we must have that $v$ maps to the point of $\m$ lying over $j =0$ in the coarse moduli space,
for $j: \m \ra \bp^1$ the coarse moduli space of $\m$.
\end{remark}

We next introduce some notation used heavily in the remainder of the proof.

\begin{notation}
	\label{notation:g-cover}
	We will assume $k = \overline{k}$.
	Suppose we have some vector bundle $\scv$ on $\m$, a point $x: \spec K
	\to \m$, and a place $v$ of $K$ so that $\delta_{\scv;v}(x) =
	\height_v(x)$.
	Note that $\m$ has coarse space given by the $j$-invariant map $j: \m \to \mathbb P^1_k$.
	Construct the universal tuning stack $\mathscr C \to \m$. 
	Let $B(G_v)$ denote the residual
	gerbe of $\mathscr C$ at the place $v$, so that we
	obtain an induced map $B(G_v) \to BG \xrightarrow{\iota} \m$
	inducing an injection $G_v \to G$ on inertia groups.
	Let $s$ denote the geometric point over $BG \to \m$.
	Then, under the identification between vector bundles on $BG$ and
	$G$ representations, $\iota^* \scv$
	can be viewed as a $G$-representation $\rho: G \ra \gl(\scv|_s)$.

Now, recall that in \autoref{remark:dicyclic}, we chose a splitting $\bz/4\bz \ra G$.
Fix a generator $1$ of $\bz/4\bz$ in $G$. This abelian group then has a
diagonalizable action on $\scv|_s$.
When we restrict $\rho$ to a $\bz/4\bz$ representation,
we obtain a decomposition $\rho|_{\bz/4\bz} \simeq \oplus_{i=0}^3 \chi_i^{\oplus
b_i}$ where $\chi_i$ are the four characters of $\bz/4\bz$ given by $\chi_i(1) =
\zeta^i$ for
$\zeta$ a fixed primitive $4$th root of unity.
\end{notation}

We can now characterize the $b_i$ appearing in the above decomposition of $\rho|_{\bz/4\bz}$.
For the proof, it will be useful to recall the notation for local heights
introduced in \autoref{section:local-heights}.

\begin{lemma}
	\label{lemma:character-values}
	Suppose $\scv$ satisfies
	$\delta_{\scv;v}(x) = \height_v(x)$ for some place $v$ at which $x$ has
	type $\mathrm{III}$ reduction.
	Under the above decomposition of $\rho|_{\bz/4\bz} \simeq \oplus_{i=0}^3
	\chi_i^{\oplus b_i}$ from \autoref{notation:g-cover}, after possibly modifying our
	choice of generator $1$ for $\bz/4\bz$, we have $b_0 = \dim \rho - 1$, $b_1= 1$ and $b_2 = b_3 = 0$.
\end{lemma}
\begin{proof}
	We first set up notation to describe the $\bz/4\bz$ representation.
	From the explicit computation of the discriminant for an elliptic curve
	$x: \spec K \ra \m$ with a closed point $v$ of type $\mathrm{III}$ reduction, we know
	$\height_v(x) = \frac{1}{4}$, by \cite[p. 365]{Silverman:Advanced}.
	Therefore, we must have $\delta_{\scv;v}(x) = 1/4$.
	Now, $x$ induces a tuning stack map $\overline{x}: \scc \to \m$ with
	coarse
	space $\pi: \scc \to C$. 
	Let $C_v$ be the local scheme at $v \in C$ and let $\scc_v =
	\pi^{-1}(C_v)$.
	Because $v$ acquires semistable reduction after a degree $4$ cyclic
	extension, but not after any smaller extension, the residual gerbe of
	$\scc$ at $v$ is
$G_v := \mathbb Z/4\mathbb Z$ and the induced map on residual gerbes $B(G_v) \to
BG$ is obtained from an injection
	$\mathbb Z/4\mathbb Z = G_v \to G$.
	Denote by $\alpha: B(G_v) \to BG \to \m$ the composite map above.
	We may view $\alpha^* \mathscr V^\vee$ as a $\mathbb Z/4\mathbb Z$
	representation. 

	We claim that under the above identification, we may view
	$\alpha^* \mathscr V^\vee$ as the direct sum of a $1$-dimensional
	faithful representation of $\bz/4\bz$ and a codimension $1$ trivial representation.
	This will complete the proof because after modifying the generator, we
	can assume the faithful representation is $\chi_1$, in which case
	$b_0 = \dim \rho - 1$, $b_1= 1$ and $b_2 = b_3 = 0$.

	To prove our claim, we will need the following assertion:
	In general, if $\scw$ is a vector bundle on $\scc_v$, 
	we assert one can read off the degree of $\scw$ from the $\mu_4$ representation
	corresponding to the restriction of $\scw$ to the residual $\mu_4$ gerbe over $v$.
	We now explain how this assertion follows from \cite[Th\'eor\`eme
	3.13]{borne:fibres-paraboliques-et-champ-des-racines},
	which shows there is a correspondence between vector bundles on the tame
	stacky curve $\scc_v$ and parabolic bundles on $C_v$. This correspondence
	preserves degree, which is shown for proper curves in \cite[Th\'eor\`eme
	4.3]{borne:fibres-paraboliques-et-champ-des-racines}, but the proof
	works equally well for localizations by restricting from the proper
	case. 
	However, the definition of parabolic degree on $C_v$,
	as in \cite[D\'efinition
	4.1]{borne:fibres-paraboliques-et-champ-des-racines}
	is then given in terms of the corresponding $\mu_4$ representation obtained by
	restricting the vector bundle on $\scc_v$ to $v$.

	We next describe how to compute the degree of the vector bundle $\scw$
	from the previous paragraph.
	Choose an isomorphism $\mu_4 \simeq \bz/4\bz$, which is possible since $k =
	\overline k$.
	Suppose the vector bundle $\scw$ on $\scc_v$ corresponds to a $\mathbb Z/4\mathbb Z$-representation $\rho'$.
	It follows from the definition of parabolic degree that
	we can express the degree of $\scw$ as follows:
	Let $\zeta$ denote a fixed primitive fourth root of unity as in
	\autoref{notation:g-cover}.
	There is some generator $g \in \mathbb Z/4 \mathbb Z$ so that
	$\rho'(g)$ has a $b_i$ dimensional eigenspace with eigenvalue $\zeta^i$, and
	$\deg \scw = \sum_{i=1}^3 \frac{b_i}{4}.$

	We now return to proving our claim.
	If $g$ is a generator of $\bz/4\bz$, we have seen any nontrivial eigenvalue of $\rho(g)$ contributes at least $1/4$ to the
	degree $\delta_{\scv;v}(x)$.
	Since $\delta_{\scv;v}(x)= 1/4$, we find
	that $\alpha^* \mathscr V^\vee$ must be the direct sum of a
	codimension $1$ trivial representation and a $1$-dimensional nontrivial
	representation.
	Moreover, that $1$-dimensional representation must be faithful, as
	otherwise $\delta_{\scv;v}(x) = 1/4$ would be a multiple of $1/2$.
\end{proof}

In order to apply \autoref{lemma:character-values} to prove
\autoref{theorem:local}, we need to verify its hypothesis holds. We verify this
in \autoref{lemma:local-iii-reduction} below.
The specific $x$ we will use is the following elliptic curve.
\begin{example}
	\label{example:iii-example}
	The magma code
\begin{verbatim}
F<t> :=FunctionField(GF(3));
E := EllipticCurve([0,t+t^2,0,t+t^2 +t^3,t^2+t^4+t^5]);
LocalInformation(E);
\end{verbatim}
shows that the elliptic curve $y^2 = x^3 + (t+t^2)x^{2}z + (t+t^2 +t^3)xz^2 +
(t^2+t^4+t^5)z^3$
has a  unique place of additive reduction $(t)$.
Further, at $(t)$, this curve has
reduction type $\mathrm{III}$ and discriminant of valuation $3$.
\end{example}
We will use the following criterion for local stacky height to agree with local Faltings
height.
\begin{lemma}
	\label{lemma:unique-local-height}
	Suppose $\scv$ is a vector bundle on $\m$ for which $\height_\scv(x) = \height(x)$ for all points $x: \spec K \ra \m$, for $K$ a finite extension of $k(t)$.
	Any point $y: \spec k(t) \to \m$ which has a unique place $v$ of additive
	reduction, i.e., a unique place with nontrivial local height, gives
	an example of a point for which $\delta_{\scv;v}(x) = \height_v(x)$.
\end{lemma}
\begin{proof}
	Let $K/k(t)$ denote an extension on which $y$ acquires semistable
	reduction and let $z: \spec K \to \m$ denote the corresponding point.
	Because $\height_\scv(z) = \height(z)$, and both
	$\height_\scv(z) = \height_\scv^{\st}(z) = \deg(K/k(t))\height^{\st}_\scv(y)$
	and 
	$\height(z) = \height^{\st}(z) = \deg(K/k(t))\height^{\st}(y)$
	\cite[Proposition 2.14]{ellenbergSZB:heights-on-stacks},
	we conclude
	$\height_\scv^{\st}(y) = \height^{\st}(y)$.
	Because we are assuming $\height_\scv^{\st}(y) + \delta_{\scv;v}(y) =
	\height_\scv(y) = \height(y) = \height^{\st}(y) + \height_{v}(y)$
	and we have shown $\height_\scv^{\st}(y)  = \height^{\st}(y)$,
	we conclude $\delta_{\scv;v}(y) = \height_{v}(y)$.
\end{proof}

The next lemma verifies the hypothesis of
\autoref{lemma:character-values}.
\begin{lemma}
	\label{lemma:local-iii-reduction}
	Suppose $\scv$ is a vector bundle on $\m$ for which $\height_\scv(x) = \height(x)$ for all points $x: \spec K \ra \m$, for $K$ a finite extension of $k(t)$.
	The point $y: \spec k(t) \to \m$ of \autoref{example:iii-example} (base
	changed from $\mathbb F_3$ to $k$) gives
	an example of a point for which $\delta_{\scv;v}(x) = \height_v(x)$
	at the place $(t)$ of additive type $\mathrm{III}$ reduction.
\end{lemma}
\begin{proof}
	Let $v = (t)$ denote the place of $k(t)$.
	Note this is the unique place of $k(t)$ at which $y$ has
	additive reduction, by \autoref{example:iii-example}.
	The lemma then follows from \autoref{lemma:unique-local-height}.
\end{proof}

We can now prove the main result, \autoref{theorem:local}.

\subsection{Proof of \autoref{theorem:local}}
\label{subsection:proof-local}
	Retain notation from \autoref{notation:g-cover}.
	By \autoref{lemma:local-iii-reduction}, 
	the hypothesis of 
	\autoref{lemma:character-values} holds, and so
	an order $4$ element of
	$G$ has a $1$-dimensional $\zeta$ eigenspace and a codimension one $1$-eigenspace.
	By \autoref{proposition:abelianization-factor},
	$\iota^*\scv$ corresponds to a representation $\rho: G \to \on{GL}(V|_s)$
	that
	splits as the direct sum of $1$-dimensional representations.
	factoring through
	$\bz/4\bz$.
	Therefore,
	$\rho|_{\mathbb Z/3\mathbb Z}$ is trivial.\qed

\bibliographystyle{alpha}
\bibliography{/home/aaron/Dropbox/master}

\end{document}